\documentclass[11pt]{amsart}
\usepackage{cases}

\theoremstyle{plain}
\newtheorem{theorem}                {Theorem}      [section]
\newtheorem*{theorem*}                {Theorem \ref{thm:appl}}
\newtheorem{proposition}  [theorem]  {Proposition}
\newtheorem{corollary}    [theorem]  {Corollary}

\newtheorem{conjecture}   {Conjecture}

\theoremstyle{definition}

\newtheorem{remark}       [theorem]  {Remark}
\newtheorem{definition}   [theorem]  {Definition}

\DeclareMathOperator{\trace}{trace} 

\DeclareMathOperator{\Div}{div} 
 \DeclareMathOperator{\riem}{Riem}
\DeclareMathOperator{\ricci}{Ricci}

\DeclareMathOperator{\rank}{rank} 
\DeclareMathOperator{\cst}{constant}
 \DeclareMathOperator{\Hess}{Hess}
 
\DeclareMathOperator{\grad}{grad}

\numberwithin{equation}{section}

\begin{document}

\title[Biharmonic and biconservative surfaces]{Biharmonic and biconservative hypersurfaces in space forms}

\author{Dorel~Fetcu}
\author{Cezar~Oniciuc}

\thanks{}

\address{Department of Mathematics and Informatics\\
Gh. Asachi Technical University of Iasi\\
Bd. Carol I, 11 \\
700506 Iasi, Romania} \email{dorel.fetcu@etti.tuiasi.ro}

\address{Faculty of Mathematics\\ Al. I. Cuza University of Iasi\\
Bd. Carol I, 11 \\ 700506 Iasi, Romania} \email{oniciucc@uaic.ro}



\begin{abstract} We present some general properties of biharmonic and biconservative submanifolds and then survey recent results on such hypersurfaces in space forms. We also propose an alternative version of a well-known result of Nomizu and Smyth for hypersurfaces, by replacing the CMC hypothesis with the more general condition of biconservativity.
\end{abstract}

\maketitle

\section{Introduction}

Suggested in $1964$ by J. Eells and J. H. Sampson (see \cite{ES}) as a generalization of harmonic maps, the biharmonic maps represent nowadays a well-established and dynamic topic in Differential Geometry.

A biharmonic map $\phi:M\rightarrow N$ between two Riemannian manifolds is a critical point of the bienergy functional
$$
E_2:C^{\infty}(M,N)\rightarrow\mathbb{R},\quad E_{2}(\phi)=\frac{1}{2}\int_{M}|\tau(\phi)|^{2}\ dv,
$$
where $M$ is compact and $\tau(\phi)=\trace\nabla d\phi$ is the tension field of $\phi$. These maps are characterized by the Euler-Lagrange equation, also known as the biharmonic equation, obtained by G. Y. Jiang in $1986$ (see \cite{J}):
\begin{equation}\label{eq:Jiang}
\tau_2(\phi)=-\Delta\tau(\phi)-\trace R^N(d\phi,\tau(\phi))d\phi=0,
\end{equation}
where $\tau_{2}(\phi)$ is the bitension field of $\phi$.

Since any harmonic map is biharmonic, our interest is to study non-harmonic biharmonic maps, which are called proper biharmonic. A biharmonic submanifold of $N$ is a biharmonic isometric immersion $\phi:M\rightarrow N$. Sometimes, throughout this paper, we will identify a submanifold as $M$ rather than mentioning the immersion $\phi$.

In Euclidean spaces, B.-Y. Chen (see \cite{C}) proposed an alternative definition of biharmonic submanifolds that coincides with the previous one when the ambient space is $\mathbb{R}^n$. B.-Y. Chen also conjectured that {\it there are no proper biharmonic submanifolds in $\mathbb{R}^n$} (see \cite{C}). On the other hand, there are many examples and classification results in spaces of non-negative curvature, especially in the Euclidian sphere $\mathbb{S}^n$.

From the theory of biharmonic submanifolds a new interesting subject, the study of biconservative submanifolds, arose and keeps gaining ground in today's mathematical research as there are many interesting examples of biconservative submanifolds, even when the biharmonic ones fail to exist. We will present details on this rather new notion in the next sections of this survey.

In this survey, we only discuss the hypersurfaces in space forms. We omit, for space reasons, other very interesting results on submanifolds in Minkowski spaces, product spaces, complex space forms, PNMC submanifolds, etc. (see, for example, \cite{ChenPNMC}, \cite{FOP}-\cite{FP}, \cite{FMO}, \cite{T4}, \cite{IS1}, \cite{IS2}, \cite{T5}, \cite{MOP}, \cite{OTU}, \cite{S1}-\cite{S5}, \cite{T6}). For the same reason, not even in the case of hypersurfaces in space forms, we will not refer to their stability properties and index.

The paper is organized as follows. In Section $2$ we present definitions and general properties of biharmonic and biconservative submanifolds in Riemannian manifolds. The third section is devoted to biharmonic and biconservative surfaces in three dimensional space forms, and it is divided accordingly into two subsections. Finally, Section $4$, also consisting of two subsections, is concerned with biharmonic and biconservative hypersurfaces (with dimension greater than two) in space forms. We give a much simpler proof for Theorem \ref{thmJHC}, of J.~H.~Chen in \cite{JHC}, on compact proper biharmonic hypersurfaces and we prove a local version of this result, Proposition \ref{p4.20}. This result can be also seen as a local version of the fact that $|A|^2=\cst$ implies CMC for compact proper biharmonic hypersurfaces in spheres, with additional hypotheses on the curvature. We also prove Theorem \ref{thmNS} which, together with Theorem \ref{thm:main}, provides an alternative version of a result of K.~Nomizu and B.~Smyth on compact CMC hypersurfaces, by replacing the CMC hypothesis with the biconservative one, which is more general. Throughout this paper a special focus is on the case when the ambient space is a Euclidean sphere. We end by posing two open problems.

\vspace{0.5cm}

\noindent{\bf Acknowledgments.} The authors would like to thank Bang-Yen Chen, Yu Fu, Matheus Vieira, \'Alvaro P\' ampano, Hiba Bibi, and Simona Nistor for carefully reading the preliminary version of our paper and for their comments and suggestions.

\vspace{0.5cm}

\noindent{\bf Conventions.} We use the following definitions and sign conventions. Consider a smooth map $\phi:M\rightarrow N$, between two Riemannian manifolds, where $M$ is assumed to be oriented and connected. In general, we will not indicate explicitly the metrics on $M$ or $N$. Then,
$$
\Delta=-\trace(\nabla^{\phi})^2 =-\trace(\nabla^{\phi}\nabla^{\phi}-\nabla^{\phi}_{\nabla})
$$
is the rough Laplacian defined on the set of all sections in $\phi^{-1}(TN)$, that is on $C(\phi^{-1}(TN))$, and $R^N$ is the curvature tensor field of $N$, given by
$$
R^N(U,V)W=[\bar\nabla_U,\bar\nabla_V]W-\bar\nabla_{[U,V]}W.
$$
Here, $\nabla^{\phi}$ denotes the pull-back connection on $\phi^{-1}(TN)$, while $\nabla$ and $\bar\nabla$ are the Levi-Civita connections on $TM$ and $TN$, respectively. Henceforth, for the sake of simplicity, we will denote all connections on various fiber bundles by $\nabla$, the difference being clear from the context. The $n$-dimensional Euclidian sphere of radius $r$ will be denoted by $\mathbb{S}^n(r)$ and, when $r=1$, by $\mathbb{S}^n$.

\section{Definitions and general properties}

In this section we briefly recall basic and general results on biharmonic and biconservative submanifolds with a focus on surfaces and hypersurfaces.

\subsection{Biconservative submanifolds}

Consider a given map
$$
\phi:M^m\rightarrow (N^n,h),
$$
where the metric $h$ is also fixed, and define a new functional on the set $\mathcal{G}$ of all Riemannian metrics on $M$ by
$$
\mathcal{F}_2:\mathcal{G}\rightarrow\mathbb{R},\quad\mathcal{F}_2(g)=E_2(\phi).
$$
Critical points of this functional are characterized by the vanishing of the stress-energy tensor of the bienergy (see \cite{LMO}). This tensor, denoted by $S_2$, was introduced in \cite{J3} as
\begin{align*}
S_2(X,Y)=&\frac{1}{2}|\tau(\phi)|^2\langle X,Y\rangle+\langle d\phi,\nabla\tau(\phi)\rangle\langle X,Y\rangle\\&-\langle d\phi(X),\nabla_Y\tau(\phi)\rangle-\langle d\phi(Y),\nabla_X\tau(\phi)\rangle
\end{align*}
and it satisfies
$$
\Div S_2=-\langle\tau_2(\phi),d\phi\rangle.
$$
We note that, for isometric immersions, $(\Div S_2)^{\sharp}=-\tau_2(\phi)^{\top}$, where $\tau_2(\phi)^{\top}$ is the tangent part of the bitension field.

\begin{definition} A map $\phi:M\rightarrow N$ is called biconservative if $\Div S_2=0$.
\end{definition}

\begin{definition} A submanifold $\phi:M\rightarrow N$ is called biconservative if the isometric immersion $\phi$ is biconservative.
\end{definition}

\begin{remark} We have the following direct consequences.

\begin{enumerate}

\item Any biharmonic map is also biconservative.

\item If $\phi:M\rightarrow N$ is a submersion, not necessarily Riemannian, then $\phi$ is a biconservative map if and only if it is biharmonic.

\item A submanifold $M$ of $N$ is biconservative if and only if $\tau_2(\phi)^{\top}=0$.

\end{enumerate}

\end{remark}

We have the following general properties of the stress-bienergy tensor of a submanifold.

\begin{proposition}
For a submanifold $\phi:M^m\to N^n$ we have:
\begin{itemize}
\item [(i)] the stress-bienergy tensor of $\phi$ is given by
$$
S_2= - \frac{m^2}{2}|H|^2 I + 2m A_H,
$$
where $A_H$ is the shape operator in the direction of $H$, the mean curvature vector field;
\item [(ii)] $\trace S_2 = m^2|H|^2\left(2-\frac{m}{2}\right)$;
\item [(iii)]the relation between the divergence of $S_2$ and the divergence of $A_H$ is given by
$$
 \Div S_2=-\frac{m^2}{2}\grad \left(|H|^2\right)+2m\Div A_H;
$$
\item [(iv)] $\left|S_2\right|^2=m^4|H|^4\left(\frac{m}{4}-2\right)+4m^2\left|A_H\right|^2$.
\end{itemize}

\end{proposition}

\subsection{Properties of biharmonic submanifolds}

Next, we shall present some non-existence and unique-continuation properties of biharmonic maps, derived from a Weitzenb\"ock formula and the fact that these maps are characterized by an elliptic equation.

\begin{theorem}[\cite{J}]
Let $\phi:M\to N$ be a smooth map, where $M$ is compact and $\riem^N\leq 0$. Then $\phi$
is biharmonic if and only if it is harmonic.
\end{theorem}

Given up to compactness, we can state the following proposition.

\begin{proposition}[\cite{O1}]
Let $\phi:M\to N$ be an isometric immersion such that $|\tau(\phi)|=constant$ and assume that $\riem^N\leq 0$. Then $\phi$ is biharmonic if and only if it is minimal.
\end{proposition}

In the case of hypersurfaces the next results hold.

\begin{theorem}[\cite{O1}]\label{thm2.7}
Let $\phi:M\to N$ be a compact hypersurface and assume that $\ricci^N\leq 0$. Then $M$ is biharmonic if and only if
it is minimal.
\end{theorem}

\begin{proposition}[\cite{O1}]\label{prop2.8}
Let $\phi:M\to N$ be a hypersurface such that $|\tau(\phi)|=constant$, where $\ricci^N\leq 0$. Then $\phi$ is biharmonic if and only if
it is minimal.
\end{proposition}

\noindent The proofs of these non-existence results are based on a classical Weitzenb\"ock formula
$$
\frac{1}{2}\Delta|\sigma|^2=\langle\Delta\sigma,\sigma\rangle-|\nabla\sigma|^2,\quad\sigma\in C(\phi^{-1}TN),
$$
where one considers $\sigma=\tau(\phi)$ and gets $\nabla\tau(\phi)=0$. Then, using the equation
$$
|\tau(\phi)|^2=-\trace\langle\cdot,\nabla_{\cdot}\tau(\phi)\rangle,
$$
which holds for isometric immersions, one obtains that $\tau(\phi)$ vanishes.

\begin{proposition}[\cite{O1}]
Let $\phi:M\to N$ be a smooth map such that $|\tau(\phi)|=constant$ and assume that there exists a point
$p\in M$ where $\rank\phi(p)\geq 2$. If $\riem^N<0$, then $\phi$ is
biharmonic if and only if it is harmonic.
\end{proposition}

Concerning the unique continuation property, we have the following results.

\begin{theorem}[\cite{VBCO},\cite{CMO}]
Let $\phi: M\rightarrow N$ be a biharmonic map. If $\phi$ is harmonic on an open subset, then it is harmonic everywhere.
\end{theorem}

\begin{theorem}[\cite{VBCO}]
\label{sphere}
Let $\phi_1,\phi_2: M\rightarrow N$ be two biharmonic maps. If they agree on an open subset, then they are identical.
\end{theorem}

\begin{theorem}[\cite{VBCO}]
Let $\phi: M\rightarrow N$ be a biharmonic map and let $P$ be a regular, closed, totally geodesic
submanifold of $N$. If an open subset of $M$ is mapped into $P$, then all of $M$
is mapped into $P$.
\end{theorem}

\begin{remark} Here, the word 'closed' is used in its topological sense, as $P$ is a closed subset of $N$.
\end{remark}

\noindent The main idea in the proofs of the unique continuation results is to define new variables such that the biharmonic equation, initially a semi-linear elliptic equation of order four, becomes a second order semi-linear  elliptic equation. Then, by making appropriate estimations, one applies the standard result of Aronszajn in \cite{A}.

An important property of constant mean curvature (CMC) proper biharmonic submanifolds in $\mathbb{S}^n$ is that $|H|$ is bounded from above by $1$. More precisely, we have the following proposition.

\begin{proposition}[\cite{Odoc}] Let $\phi:M^m\rightarrow\mathbb{S}^n$ be a CMC proper biharmonic submanifold. Then $|H|\in(0,1]$ and, moreover, $|H|=1$ if and only if $M$ lies in the small hypersphere $\mathbb{S}^{n-1}(1/\sqrt{2})$ as a minimal submanifold.
\end{proposition}

A link between proper biharmonic immersions in spheres and maps of finite type, in the sense of B.-Y. Chen, was established by the next result.

\begin{theorem}[\cite{BMO},\cite{BMO13},\cite{LOJPN}]
Let $\phi : M^m \rightarrow \mathbb{S}^{n}$ be a proper biharmonic immersion. Denote by $\psi = i\circ\phi : M \to \mathbb{R}^{n+1}$ the isometric immersion of $M$ in $\mathbb{R}^{n+1}$, where $i : \mathbb{S}^{n}\rightarrow \mathbb{R}^{n+1}$ is the canonical inclusion map. Then
\begin{enumerate}
\item[(i)] The map $\psi$ is of $1$-type if and only if $|H|=1$. In this case, $\psi = \psi_0 + \psi_{t_1}$, with $\Delta \psi_{t_1} = 2m \psi_{t_1}$, $\psi_{0}$ is a constant vector. Moreover, $\langle \psi_0 , \psi_{t_1} \rangle =0$ at any point, $|\psi_{0}| =|\psi_{t_1}| = 1/\sqrt{2}$ and
$\phi_{t_1}:M\rightarrow \mathbb{S}^{n-1}\left(1/\sqrt{2}\right)$ is a minimal immersion.
\item[(ii)] The map $\psi$ is of $2$-type if and only if $|H|$ is constant, $|H|\in (0,1)$. In this case $\psi = \psi_{t_1} + \psi_{t_2}$, with $\Delta \psi_{t_1} = m(1-|H|)\psi_{t_1}$,
$\Delta \psi_{t_2} = m(1+|H|)\psi_{t_2}$ and
\begin{align*}
\psi_{t_1} &= \tfrac{1}{2} \psi + \tfrac{1}{2|H|} H \, , \,  \psi_{t_2} = \tfrac{1}{2} \psi - \tfrac{1}{2|H|} H .
\end{align*}
Moreover, $\langle \psi_{t_1} , \psi_{t_2} \rangle =0$, $|\psi_{t_1}| =|\psi_{t_2}| = 1/\sqrt{2}$ and
$$ \phi_{t_i} : M \rightarrow \mathbb{S}^{n}\left(\frac{1}{\sqrt{2}}\right), \quad i=1,2,$$
are harmonic maps with constant density energy.
\end{enumerate}
\end{theorem}

This theorem was the starting point for the studies on biharmonic surfaces with constant Gaussian curvature in space forms $N^n(c)$ (see \cite{FO-DGA}, \cite{LOJPN}), where one used the work of Y. Myiata (see \cite{M}) to obtain classification results.

\subsection{Properties of biconservative submanifolds}

The biconservative submanifolds do not have properties similar to those of biharmonic ones in the previous subsection as they are not characterized by an elliptic equation. However, they do have some very interesting properties of their own, especially when talking about biconservative surfaces. Moreover, the class of biconservative maps is richer than that of biharmonic maps and, in many situations when proper biharmonic submanifolds do not exist, we have examples of biconservative submanifolds.

One of these properties can be viewed as a generalization of a classical Hopf's result in higher codimension, and it shows that the notion of biconservativity is a quite natural one.

\begin{theorem}[\cite{MOR2},\cite{N}]\label{N}
Let $\phi:M^2\rightarrow N^n$ be a CMC surface. Then the following properties are equivalent:
\begin{itemize}
  \item [(i)] $M$ is biconservative;
  \item [(ii)] $\langle A_H\left(\partial_z\right),\partial_z \rangle$ is holomorphic;
  \item [(iii)] $A_H$ is a Codazzi tensor field.
\end{itemize}
\end{theorem}

\noindent We note that the above theorem follows by using the properties of any divergence-free symmetric
tensor field of type $(1, 1)$ defined on a Riemannian surface.

\begin{remark} In the case of biharmonic surfaces, the fact that $\langle A_H\left(\partial_z\right),\partial_z \rangle$ is holomorphic if and only if $M$ is CMC was proved in \cite{LO}.
\end{remark}

\begin{remark} We can say that $\langle A_H(\partial_z),\partial_z\rangle$ is a generalization of the Hopf's function as in $3$-dimensional space forms we have the following. If $M^2$ is a surface in $N^3(c)$ which is a topological sphere, then it is CMC, i.e., $|H|=\cst$, if and only if $\langle A_H(\partial_z),\partial_z\rangle$ is holomorphic (see \cite{MOR2}). We note that for any surface in $N^3(c)$, $\langle A_H(\partial_z),\partial_z\rangle$ is holomorphic when $M$ is CMC, but, in general, the converse does not hold. All non-CMC surfaces with $\langle A_H\left(\partial_z\right),\partial_z \rangle$ holomorphic must have the curvature equal to $c$ (but they are not umbilical). In \cite{MOR2}, there were found all surfaces in $\mathbb{R}^3$ with $\langle A_H\left(\partial_z\right),\partial_z \rangle$ holomorphic that are not CMC, and they are given by
$$
X_{\alpha}(u,v)=(u\cos v,u\sin v,\alpha u),
$$
where $\alpha$ is a positive real constant. In $\mathbb{S}^3$ such surfaces are given by
$$
X_{\alpha}(u,v)=\left(\frac{\cos u}{\alpha}\cos v,\frac{\sin u}{\alpha}\cos v,\frac{\sqrt{\alpha^2-1}}{\alpha}\cos v,\sin v\right),
$$
where $\alpha$ is a real constant, $\alpha>1$ (see \cite{N}).
\end{remark}

\begin{remark}
If $\phi:M^2\rightarrow N^n$ is a non-pseudo-umbilical $CMC$ biconservative surface, then the set of pseudo-umbilical points has no accumulation points. Also, if $M^2$ is a $CMC$ biconservative surface and a topological sphere, then it is pseudo-umbilical (see \cite{MOR2}); this should be compared with a classical result: a $PMC$ surface $M^2$, i.e., a surface satisfying $\nabla^{\perp}H=0$, of genus zero in a space form is pseudo-umbilical (see \cite{DH}).
\end{remark}

By using Theorem \ref{N} and the CMC hypothesis, around non-pseudo-um\-bi\-li\-cal points, one can obtain an explicit form of the metric on the surface and of the shape operator $A_H$. It follows that the CMC biconservative surfaces with no pseudo-umbilical points are globally conformally flat. Then, one can prove the following proposition.

\begin{proposition}[\cite{N}]\label{past}
Let $\phi:M^2\to N^n$ be a CMC biconservative surface. Assume that $M$ is compact and does not have pseudo-umbilical points. Then $M$ is a topological torus and, moreover, if $K\geq 0$ or $K\leq 0$, we have $\nabla A_H=0$ and $K=0$.
\end{proposition}

\begin{remark} One can compare this result with Theorem $5.4$ in \cite{COTAMS}.
\end{remark}

We can also deduce that if $M^2$ is a CMC biconservative surface satisfying certain additional hypotheses, then, up to a global conformal diffeomorphism, $M^2$ either is $\mathbb{R}^2$, or a cylinder, or a torus.

\begin{proposition} Let $\phi:M^2\rightarrow N^n$ be an oriented complete CMC biconservative surface. Denote by $\lambda_1$ and $\lambda_2$ the eigenvalue functions of $A_H$ and assume that $\mu=\lambda_1-\lambda_2$ is positive with $\inf\mu=\mu_0>0$. Then $M$ is parabolic.
\end{proposition}

\noindent The proof is a direct consequence of Theorem 4.5 in \cite{N} (see also Theorem 12 in \cite{LO}) as $\mu g$ is a flat complete metric globally conformally equivalent to $g$.

Another important property of biconservative surfaces is given by the following result.

\begin{theorem}[\cite{N}]\label{t222}
Let $\phi:M^2\to N^n$ be a compact $CMC$ biconservative surface. If the Gaussian curvature $K$ of the surface is nonnegative, then $\nabla A_H=0$ and $M$ is flat or pseudo-umbilical.
\end{theorem}

\noindent The proof of this result relies on a formula for the Laplacian of a divergence-free symmetric tensor of type $(1,1)$ defined on a Riemannian surface, not necessarily satisfying the Codazzi condition.

\begin{remark} We note that Proposition \ref{past}, only in the case $K\geq 0$, can be viewed as a consequence of Theorem \ref{t222}.
\end{remark}

In codimension two we have the following rigidity result, since any PMC submanifold in a space form is biconservative.

\begin{theorem}[\cite{MOR2}]
Let  $\phi:M^2\rightarrow N^4(c)$ be a CMC biconservative surface in a space form of constant sectional curvature $c\neq 0$. Then $M^2$ is PMC.
\end{theorem}

In the special case of biconservative surfaces in $\mathbb{R}^4$, the situation is a bit less rigid.

\begin{proposition}[\cite{MOR2}]
Let $\phi:M^2\rightarrow\mathbb{R}^4$ be a biconservative surface with constant mean curvature different from zero, which is not PMC. Then, locally, the surface is given by
\begin{equation*}
X(u,v)=(\gamma(u),v+a)=(\gamma^1(u),\gamma^2(u),\gamma^3(u),v+a),\quad a\in\mathbb{R} \,\, ,
\end{equation*}
where $ \gamma:I\rightarrow \mathbb{R}^3$ is a curve in $\mathbb{R}^3$ parametrized by arc-length, with constant non-zero curvature, and non-zero
torsion.
\end{proposition}

\subsection{Characterization formulas for biharmonic submanifolds}

As we have seen, the tension field $\tau(\phi)$ of a map $\phi:M\rightarrow N$ plays a central role in the biharmonic equation \eqref{eq:Jiang}. In the case of a submanifold $M^m$, we have $\tau(\phi)=mH$ and the tension field is a normal vector field. Therefore it is just natural to identify the tangent and normal parts of \eqref{eq:Jiang} in order to better characterize bihamornic submanifolds.

The simplest case occurs when the ambient manifold is a space form $N^n(c)$. In this context, the decomposition of $\tau_2(\phi)$ was written for the first time for any constant sectional curvature $c$ in \cite{BMO}, although for $c=0$, it had been given explicitly in \cite{BYC}, \cite{BYC2}, for $c=1$ in \cite{O1}, and for $c=-1$ in \cite{CMO}.

\begin{theorem} A submanifold $\phi:M^m\rightarrow N^n(c)$ is biharmonic if and only if
$$
\left\{
\begin{array}{l}
\ \Delta^\perp {H}+\trace B(\cdot,A_{H}\cdot)-mc\,{H}=0,
\vspace{2mm}
\\
\ 4\trace A_{\nabla^\perp_{(\cdot)}{H}}(\cdot)
+m\grad(|H|^2)=0,
\end{array}
\right.
$$
where $B$ is the second fundamental form, $\nabla^\perp$ and $\Delta^\perp$ the connection and the Laplacian
in the normal bundle, respectively.
\end{theorem}

When $M$ is a hypersurface, we can consider the mean curvature function $f=(1/m)\trace A$, where $A=A_{\eta}$, $\eta$ being a unit section in the normal bundle, and we have the following corollary.

\begin{corollary} Let $M^m$ be a hypersurface in a space form $N^{m+1}(c)$. Then $M$ is biharmonic if and only if
$$
\left\{
\begin{array}{l}
\ \Delta f+(|A|^2-mc)f=0,
\vspace{2mm}
\\
\ 2A(\grad f)+mf\grad f=0,
\end{array}
\right.
$$
\end {corollary}

In this last form, for $c=1$, the decomposition of $\tau_2(\phi)$ was used in \cite{JHC}.

The biharmonic equation was also decomposed in its tangent and normal parts when the ambient space is a complex space form (see \cite{FLMO}) or a Sasakian space form (see \cite{FO-DGA}).

Finally, in the case when the ambient space is an arbitrary Riemannian manifold, the splitting of the biharmonic equation was obtained in \cite{LMO}.

\begin{theorem}
A submanifold $M^m$ in a Riemannian manifold $N$ is biharmonic if and only if
$$
\begin{cases}
\Delta^{\perp}H+\trace B(\cdot,A_H\cdot)+\trace(R^N(\cdot,H)\cdot)^{\perp}=0\\
4\trace A_{\nabla^{\perp}_{(\cdot)}H}(\cdot)+m\grad(|H|^2)+4\trace(R^N(\cdot,H)\cdot)^{\top}=0,
\end{cases}
$$
where $R^N$ is the curvature tensor of $N$. Moreover, the tangent part can be written as
$$
4\trace(\nabla A_H)(\cdot,\cdot)-m\grad(|H|^2)=0.
$$
\end{theorem}

While these formulas were used to study biharmonic hypersurfaces in a Riemannian manifold, for the first time, in \cite{O1} (see Theorem \ref{thm2.7} and Proposition \ref{prop2.8}), they were explicitly written for this case in \cite{O-Pac}.

\begin{theorem} Let $M^m$ be a hypersurface in a Riemannian manifold $N$. Then $M$ is biharmonic if and only if
$$
\begin{cases}
\Delta f+(|A|^2-\ricci^N(\eta,\eta))f=0\\
2A(\grad f)+mf\grad f-2f(\ricci^N(\eta))^{\top}=0,
\end{cases}
$$
where $\eta$ is a unit normal vector field.
\end{theorem}

Many examples of proper-biharmonic submanifolds in spheres are provided by the following results.

\begin{theorem}[\cite{BO}]
Let $n_1, n_2$ be two positive integers such that $n_1+n_2=n-1$, and
let $M_1$ be a submanifold in $\mathbb{S}^{n_1}(1/\sqrt 2)$ of
dimension $m_1$, with $0 \leq m_1 \leq n_1$, and let $M_2$ be a
submanifold in $\mathbb{S}^{n_2}(1/\sqrt 2)$ of dimension $m_2$,
with $0 \leq m_2 \leq n_2$. Then $M_1\times M_2$ is proper
biharmonic in $\mathbb{S}^n$ if and only if
\begin{equation*}
\left\{
\begin{array}{ll}
m_1\neq m_2,\quad\textnormal{or}\quad |\tau(\phi_1)|>0\\
\tau_2(\phi_1)+2(m_2-m_1)\tau(\phi_1)=0\\
\tau_2(\phi_2)-2(m_2-m_1)\tau(\phi_2)=0\\
|\tau(\phi_1)|=|\tau(\phi_2)|=\cst,
\end{array}
\right.
\end{equation*}
where $\phi_1:M_1\to \mathbb{S}^{n_1}(1/\sqrt 2)$ and
$\phi_2:M_2\to \mathbb{S}^{n_2}(1/\sqrt 2)$ are the associated isometric immersions.
\end{theorem}

\begin{remark} We note that here we also correct a small inaccuracy in the original version of the theorem (see \cite{CMO}).
\end{remark}

\begin{corollary}[\cite{CMO}]
Let $M^{m_1}_1$ and $M^{m_2}_2$ be two minimal submanifolds in the spheres $\mathbb{S}^{n_1}(1/\sqrt{2})$ and $\mathbb{S}^{n_2}(1/\sqrt{2})$, respectively, with $n_1+n_2=n$. Then the product $M_1\times M_2$ is a proper biharmonic submanifold of $\mathbb{S}^{n+1}$ if and only if $m_1\neq m_2$.
\end{corollary}

\subsection{Characterization formulas for biconservative submanifolds}

Since a submanifold is biconservative if and only if the tangent part of its bitension field vanishes, from the splitting of the biharmonic equation and looking only to the tangent part, one obtains formulas characterizing the biconservativity.

\begin{proposition}
Let $M^m$ be a submanifold in a Riemannian manifold $N^n$. Then the following properties are equivalent:
\begin{itemize}
    \item [(i)] $M$ is biconservative;
    \item [(ii)] $\trace A_{\nabla^\perp_{(\cdot)} H}(\cdot)+\trace(\nabla A_H)(\cdot,\cdot) +\trace \left( R^N(\cdot,H)\cdot\right)^\top=0$;
    \item [(iii)] $4\trace A_{\nabla^\perp_{(\cdot)} H}(\cdot) + m\grad\left(|H|^2\right)+4\trace \left(R^N(\cdot,H)\cdot\right)^\top=0$;
    \item [(iv)] $4\trace(\nabla A_H)(\cdot,\cdot)-m\grad\left(|H|^2\right)=0$.
\end{itemize}
\end{proposition}

From this proposition, we have the following consequences, some of which are quite straightforward.

\begin{proposition}\label{p232}
Let $M^m$ be a submanifold of a Riemannian manifold $N^n$. If $\nabla A_H=0$, then $M$ is biconservative.
\end{proposition}

\begin{remark} A converse of Proposition \ref{p232} is given by Theorem \ref{t222}.
\end{remark}

As we have already mentioned, we have the following consequence.

\begin{proposition}
Let $M^m$ be a submanifold of a Riemannian manifold $N^n$. Assume that $N$ is a space form and $M$ is $PMC$. Then $M$ is biconservative.
\end{proposition}

\begin{proposition}[\cite{BMO13}]
Let $M^m$ be a pseudo-umbilical submanifold of a Riemannian manifold $N^n$ with $m\neq 4$. Then $M$ is biconservative if and only if it is a CMC submanifold.
\end{proposition}

If the particular case of hypersurfaces in space forms, we have the following result.

\begin{proposition}
A hypersurface $M^m$ in a space form $N^{m+1}(c)$ is biconservative if and only if
\begin{equation}\label{eq:bicons}
A(\grad f)=-\frac{m}{2}f\grad f.
\end{equation}
\end{proposition}

\begin{corollary}
Any $CMC$ hypersurface in $N^{m+1}(c)$ is biconservative.
\end{corollary}

\begin{corollary}[\cite{SN-Thesis}]
Let $M^m$ be a biconservative hypersurface in $N^{m+1}(c)$. Then
$$
mf\Delta f-3m|\grad f|^2-2\langle A,\Hess f\rangle=0.
$$
\end{corollary}

\begin{remark} In \cite{SN-Thesis}, there was used a different definition for the mean curvature function $f$, i.e., there $f=\trace A$ instead of $f=(1/m)\trace A$.
\end{remark}

Moreover, when dealing with biconservative hypersurfaces in space forms, the two distributions determined by $\grad f$ are completely integrable, as showed by the next theorem.

\begin{theorem}[\cite{SN-Thesis}]
Let $M^m$ be a biconservative hypersurface in $N^{m+1}(c)$ with $\grad f\neq 0$ at any point of $M$. Then the distribution $\mathcal{D}$, orthogonal to that determined by $\grad f$, is completely integrable. Moreover, any integral manifold of maximal dimension of $\mathcal{D}$ has flat normal connection as a submanifold in $N^{m+1}(c)$.
\end{theorem}

\begin{remark}
The last result actually extends the one obtained in \cite{HV95}, in the case of Euclidean space.
\end{remark}

\section{Surfaces in three dimensional space forms}

\subsection{Biharmonic surfaces}

The first result in this section is a non-existence result, that was proved in \cite{CI} and \cite{J2} when working in the Euclidean space, and then in \cite{CMO} in the case of space forms with negative constant sectional curvature.

\begin{theorem} Any biharmonic surface in a space form $N^3(c)$, with $c\leq 0$, is minimal.
\end{theorem}

The situation is different when one considers biharmonic surfaces in the Euclidian sphere, as shown by the following theorem in \cite{CMO1}.

\begin{theorem} Let $\phi:M\rightarrow\mathbb{S}^3$ be a proper-biharmonic surface. Then $\phi(M)$ is an open part of the small hypersphere $\mathbb{S}^2(1/\sqrt{2})$. If, moreover, $M$ is complete, then $\phi(M)=\mathbb{S}^2(1/\sqrt{2})$ and $\phi$ is an embedding.
\end{theorem}

\subsection{Biconservative surfaces}

As we have seen, CMC surfaces in a $3$-di\-men\-sio\-nal space form are biconservative and, therefore, from this point of view, the study of non-CMC biconservative surfaces is more interesting. The maximal biconservative surfaces with $\grad f\neq 0$ at any point are called standard biconservative surfaces, and the domains of their defining immersions endowed with the induced metrics are called abstract standard biconservative surfaces.

The explicit parametric equations of standard biconservative surfaces in $N^3(c)$ were obtained in \cite{CMOP}, while they had already been found, in a slightly different manner, when working in $\mathbb{R}^3$ (see \cite{HV95}). In the hyperbolic three dimensional space, one of the possible situations was omitted in \cite{CMOP}, the one corresponding to $C_{-1}=0$, as we will see later, but it was then treated in \cite{Fu} (see also \cite{NO-NLA}).

We present here only the case when the ambient space is the Euclidean sphere (that is the sectional curvature $c$ is equal to 1).

\begin{theorem}[\cite{CMOP}]\label{thm:CMOP33}
Let $M^2$ be a biconservative surface in $\mathbb{S}^3$ with $(\grad f)(p)\neq 0$ at any point $p\in M$. Then, the surface, viewed in $\mathbb{R}^4$, can be parametrized locally by
\begin{equation}\label{eq:YC1tilda}
Y_{\tilde{C}_1}(u,v)=\sigma(u)+\frac{4\kappa(u)^{-3/4}}{3\sqrt{\tilde{C}_1}}\left( \overline{f}_1 (\cos v -1)+\overline{f}_2 \sin v\right),
  \end{equation}
where $\tilde{C}_1\in\left(64/\left(3^{5/4}\right),\infty\right)$ is a positive constant; $\overline{f}_1, \overline{f}_2\in\mathbb{R}^4$ are two constant orthonormal vectors; $\sigma(u)$ is a curve parametrized by arc-length that satisfies
\begin{equation}
\label{eq:sigma_prod_scal}
  \langle \sigma(u),\overline{f}_1\rangle = \frac{4\kappa(u)^{-3/4}}{3\sqrt{\tilde{C}_1}}, \qquad \langle \sigma(u),\overline{f}_2\rangle=0,
\end{equation}
and, as a curve in $\mathbb{S}^2$, its curvature $\kappa=\kappa(u)$ is a positive non-constant solution of the following $ODE$
\begin{equation}\label{k''k}
\kappa^{\prime\prime}\kappa=\frac{7}{4}\left(\kappa^\prime\right)^2+\frac{4}{3}\kappa^2-4\kappa^4
\end{equation}
such that
\begin{equation}\label{eq:prime_integ}
\left(\kappa^\prime\right)^2=-\frac{16}{9}\kappa^2-16\kappa^4+\tilde{C}_1\kappa^{7/2}.
\end{equation}
\end{theorem}

\begin{remark}\label{r3.4}
The curve $\sigma$ lies in the totally geodesic $\mathbb{S}^2=\mathbb{S}^3\cap\Pi$, where $\Pi$ is the linear hyperspace of $\mathbb{R}^4$ orthogonal to $\overline{f}_{2}$. The constant $\tilde{C}_1$ determines uniquely the curvature $\kappa$, up to a translation and change of sign of $u$, and then $\kappa$, $\overline{f}_1$ and $\overline{f}_2$ determine uniquely the curve $\sigma$. Thus, up to isometries of $\mathbb{S}^3$, we have a one parametric family of standard biconservative surfaces in $\mathbb{S}^3$ indexed by $\tilde{C}_1$. Moreover, $\sigma$ is a geodesic of $M$.
\end{remark}

\begin{remark} From \eqref{k''k} it follows that $\sigma$ is a critical point of the curvature energy functional $\Theta(\gamma)=\int_{\gamma}\kappa^{1/4}$, where $\gamma$ is a curve in $\mathbb{S}^2$ parametrized by arc-length, in the sense of classical elastic curves (see \cite{MP}).
\end{remark}

None of these standard biconservative surfaces in space forms $N^3(c)$ is complete. Therefore, the next step was finding complete non-CMC biconservative surfaces with $\grad f\neq 0$ on an open subset.

The $c=0$ case proved to be the simplest one and one was able to construct such surfaces by using both an extrinsic and an intrinsic approach. Thus, working extrinsically, i.e., with the explicit parametric equations, as a standard biconservative surface in $\mathbb{R}^3$ is a rotational surface with the boundary (in its topogical closure) connected and given by one circle, one glued two profile curves thus obtaining a regular closed surface. In this case, the gluing was made along the shared boundary that becomes a geodesic on the resulting surface, where $\grad f=0$. In the intrinsic way, one first constructed an abstract simply connected complete surface and then the respective non-CMC biconservative immersion. We note that to obtain the abstract surface, one used isothermal coordinates and, again, a gluing process of two abstract standard biconservative surfaces. We note that this surface cannot be factorized to a (non-flat) torus.

\begin{theorem}[\cite{SN-JGP},\cite{SN-Thesis}]
Let $\left(\mathbb{R}^2,g_{C_0}=C_0 \left(\cosh u\right)^6\left(du^2+dv^2\right)\right)$ be a surface, where $C_0\in\mathbb{R}$ is a positive constant. Then we have:
\begin{itemize}
\item[(i)] the metric on $\mathbb{R}^2$ is complete;
\item[(ii)] the Gaussian curvature is given by
           $$
            K_{C_0}(u,v)=K_{C_0}(u)=-\frac{3}{C_0\left(\cosh u\right)^8}<0,\quad  K^\prime_{C_0}(u)=\frac{24 \sinh u}{C_0\left(\cosh u\right)^9},
           $$
           and therefore $\grad K_{C_0}\neq 0$ at any point of $\mathbb{R}^2\setminus Ov$;

\item[(iii)] the immersion $\phi_{C_0}:\left(\mathbb{R}^2,g_{C_0}\right)\to \mathbb{R}^3$ given by
    $$
    \phi_{C_0}(u,v)=\left(\sigma_{C_0}^1(u)\cos (3v), \sigma_{C_0}^1(u)\sin (3v), \sigma_{C_0}^2(u)\right)
    $$
    is biconservative in $\mathbb{R}^3$, where
    $$
    \sigma_{C_0}^1(u)=\frac{\sqrt{C_0}}{3}\left(\cosh u\right)^3, \quad
    \sigma_{C_0}^2(u)=\frac{\sqrt{C_0}}{2}\left(\frac{1}{2}\sinh (2u)+u\right), \qquad u \in \mathbb{R}.
    $$
\end{itemize}
\end{theorem}

To intrinsically approach the $c\neq 0$ cases, as well as in the $c=0$ case actually, one needs to use the following intrinsic properties of the standard biconservative surfaces.

\begin{theorem}[\cite{CMOP}]\label{thm:CMOP}
Let $\phi:M^2\rightarrow N^3(c)$ be a biconservative surface with $\grad f\neq 0$ at any point of $M$. Then the Gaussian curvature $K$ satisfies
\begin{itemize}
\item [(i)]
    $$
    K=\det A+c=-3f^2+c;
    $$
\item [(ii)] $c-K>0$, $\grad K\neq 0$ at any point of $M$, and the level curves of $K$ are circles in $M$ with constant curvature
\begin{equation*}
\kappa=\frac{3|\grad K|}{8(c-K)};
\end{equation*}
\item [(iii)]
$$
(c-K)\Delta K-|\grad K|^2-\frac{8}{3}K(c-K)^2=0.
$$
\end{itemize}
\end{theorem}

\begin{remark} If $M^2$ is a biconservative surface in $N^3(c)$ with $\grad f\neq 0$ at any point, then it is a linear Weingarten surface. Indeed, from $A(\grad f)=-f\grad f$, we have $3\lambda_1+\lambda_2=0$, where $\lambda_1$ and $\lambda_2$ are the principal curvatures of $M^2$.
\end{remark}

Now, we present a uniqueness result for biconservative surfaces with nowhere vanishing $\grad f$.

\begin{theorem}[\cite{FNO}]\label{thm:uniqueness}
Let $\left(M^2,g\right)$ be an abstract surface and $c\in\mathbb{R}$ an arbitrarily fixed constant. If $M$ admits two biconservative immersions in $N^3(c)$ such that the gradients of their mean curvature functions are different from zero at any point of $M$, then the two immersions differ by an isometry of $N^3(c)$.
\end{theorem}

Even if the notion of a biconservative submanifold belongs, obviously, to extrinsic geometry, in the particular case of biconservative surfaces in $N^3(c)$ one can give an intrinsic characterization of such surfaces.

\begin{theorem}[\cite{FNO}]\label{thm:char}
Let $\left(M^2,g\right)$ be an abstract surface. Then $M$ can be locally isometrically embedded in a space form $N^3(c)$ as a biconservative surface with the gradient of the mean curvature different from zero everywhere if and only if the Gaussian curvature $K$ satisfies $c-K(p)>0$, $(\grad K)(p)\neq 0$, for any $p\in M$, and its level curves are circles in $M$ with constant curvature
$$
\kappa=\frac{3|\grad K|}{8(c-K)}.
$$
\end{theorem}

Even more, one can prove that any biconservative immersion from $(M^2,g)$ in $N^3(c)$ must satisfy $\grad f\neq 0$, and so one obtains the next result.

\begin{theorem}[\cite{SN-Thesis}]\label{thm:reformulate}
Let $\left(M^2,g\right)$ be an abstract surface and $c\in\mathbb{R}$ an arbitrarily given constant. Assume that $c-K>0$ and $\grad K\neq0$ at any point of $M$, and the level curves of $K$ are circles in $M$ with constant curvature
$$
\kappa=\frac{3|\grad K|}{8(c-K)}.
$$
Then, locally, there exists a unique biconservative embedding $\phi:\left(M^2,g\right)\to N^3(c)$. Moreover, the mean curvature function is positive and its gradient is different from zero at any point of $M$.
\end{theorem}

\begin{remark} For a given $C$, there exists a one parametric family of abstract standard biconservative surfaces indexed by $C$, and their metrics can be written in a very explicit way.
\end{remark}

In the case when the ambient space is the Euclidian sphere, the extrinsic approach proved to be quite difficult, as it implies the gluing of a large number (infinite, in general) of standard biconservative surfaces. When working intrinsically, one first used a gluing procedure, to obtain a closed rotational surface in $\mathbb{R}^3$ (and so complete), with a periodic profile curve, and then, from that surface (or its universal cover $(\mathbb{R}^2,g_{C_1})$, that can be factorized to a non-flat torus), one constructed the desired non-CMC biconservative immersion (see \cite[Theorem 4.18]{SN-JGP}).

The hyperbolic case ($c=-1$) is rather similar to the $c=0$ one. In an extrinsic approach one has to glue two standard biconservative surfaces to obtain complete biconservative surfaces which are not CMC (see \cite{NO-NLA}). In the intrinsic manner, in \cite{NO-NLA} one proposes a new method that can be easily adapted to the other two cases, when $c=0$ or $c=1$. The main idea is to rewrite the metric on an abstract standard biconservative surface as
$$
g(u,v)=h^2(u)dv^2+du^2,\quad u>0,\quad v\in\mathbb{R},
$$
where $h(u)>0$ and $\lim_{u\rightarrow 0}h(u)\in\mathbb{R}_{+}^{\ast}$. Then one glues two such abstract standard biconservative surfaces along the axis $(Ov)$ to obtain an abstract simply connected complete surface $(\mathbb{R}^2,g_{C_{-1}})$ that cannot be factorized to a torus. To prove the existence of the non-CMC biconservative immersion from this surface in $\mathbb{H}^3$, one constructs a suitable shape operator and then uses the Fundamental Theorem of Surfaces in space forms. The difficult part of this construction is the fact that  $\grad K$ vanishes along the axis $(Ov)$ and, therefore, one cannot directly apply Theorem \ref{thm:char}.

We note that all abstract simply connected complete surfaces $(\mathbb{R}^2,g_{C_{-1}})$, $(\mathbb{R}^2,g_{C_{0}})$ and $(\mathbb{R}^2,g_{C_{1}})$ are globally conformally equivalent to $\mathbb{R}^2$.

With all these results in mind, two interesting issues can be raised.

\vspace{0.5cm}

\noindent {\bf Problem 1.} Let $\phi:M^2\rightarrow N^3(c)$ be a simply connected complete biconservative surface which is not CMC. Then, up to an isometry of the domain or the codomain, $M$ and $\phi$ are those given in \cite{SN-JGP}, when $c=0$ or $c=1$, and in \cite{NO-NLA}, when $c=-1$.

\vspace{0.2cm}

\noindent {\bf Problem 2.} Any compact biconservative surface in $N^3(c)$ is CMC.

\vspace{0.5cm}

Very recently, Problem 1, was positively answered in \cite{NO20}, and one the most important steps of the proof consisted in proving that (for the $c=-1$ case) from a given abstract surface $(\mathbb{R}^2,g_{c_{-1}})$ there exists a unique biconservative immersion in $\mathbb{H}^3$. Moreover, it is not CMC. The proof of this step relies, among others, on Theorem \ref{thm:reformulate}, but here $\grad K$ vanishes along $(Ov)$. This intermediary result in the proof of Problem 1 also holds when $c=0$ or $c=1$.

Now, Problem 2 can be easily solved for $c=0$ or $c=-1$. More precisely, one considers the universal cover $\pi:\tilde M^2\rightarrow M^2$ of a compact non-CMC biconservative surface $M^2$ and thus one obtains a surface $\tilde\phi:\tilde M^2\rightarrow\mathbb{R}^3$ or $\tilde\phi:\tilde M^2\rightarrow\mathbb{H}^3$, which is simply connected, complete, non-CMC and biconservative. But then, the immersion $\tilde\phi$ would have to be one of those given in \cite{SN-JGP} and \cite{NO-NLA}, and, in particular, $\tilde{M}^2=(\mathbb{R}^2,g_{C_{-1}})$ or $\tilde{M}^2=(\mathbb{R}^2,g_{C_{0}})$, for some constants $C_{-1}$ and $C_{0}$. But this is a contradiction as $\tilde{M}^2=(\mathbb{R}^2,g_{C_{-1}})$ or $\tilde{M}^2=(\mathbb{R}^2,g_{C_{0}})$ can be factorized only to a cylinder (which is
not compact). The case $c=1$ in Problem 2 cannot be solved as easily as the other two cases, even if the abstract surface $(\mathbb{R}^2,g_{C_{1}})$ does factorize to a non-flat torus, and the author could not prove whether (some of) the immersions in \cite{SN-JGP} are double-periodic or not.

In fact, Problem 2 was completely solved in \cite{MP}. Without addressing Problem 1, the authors prove that while in $N^3(c)$, with $c=-1$ or $c=0$, all compact biconservative surfaces are CMC, in the case $c=1$, there exists an entire family of compact non-CMC biconservative surfaces. In other words, among the immersions given \cite{SN-JGP} in the case of the sphere, there exists a family of double-periodic ones. To prove the existence of this family of surfaces in $\mathbb{S}^3$, the authors used the known fact that the standard biconservative surface in the sphere is rotational and that the curvature of its profile curve, that lies in a sphere $\mathbb{S}^2$ totally geodesic in $\mathbb{S}^3$,  satisfies a certain equation (see Theorem \ref{thm:CMOP33}). They proved, using among other results the Poincare-Bendixon Theorem, that the maximally defined solutions of this equation are defined on the entire real axis $\mathbb{R}$ and are periodic. Now, having these periodic curvature functions and using closure conditions from the theory of elastic curves (see \cite{AGM}), the authors showed that there exist periodic profile curves in $\mathbb{S}^2$ having as curvatures some of these periodic curvature functions. Obviously, the corresponding surface to a periodic profile curve is compact, biconservative and non-CMC.

\section{Hypersurfaces in space forms}

\subsection{Biharmonic hypersurfaces}

All results obtained until now for biharmonic hypersurfaces in spheres are partial positive answers to the following two conjectures, which have been proposed in \cite{BMO}.

\vspace{0.5cm}

\begin{conjecture}\label{c1} Any proper biharmonic hypersurface in $\mathbb{S}^{m+1}$ is either an open part of $\mathbb{S}^{m}(1/\sqrt{2})$, or an open part of $\mathbb{S}^{m_1}(1/\sqrt{2})\times\mathbb{S}^{m_2}(1/\sqrt{2})$, where $m_1+m_2=m$ and $m_1\neq m_2$.
\end{conjecture}

\vspace{0.2cm}

\begin{conjecture}\label{c2} Any proper biharmonic submanifold in $\mathbb{S}^{m+1}$ is CMC.
\end{conjecture}

\vspace{0.5cm}

Obviously, Conjecture \ref{c2} (for hypersurfaces) is weaker than Conjecture \ref{c1}, but to directly prove the first conjecture seems to be quite a complicated task. Until now, a proof for the second conjecture, which may be viewed as a step in the proof of Conjecture \ref{c1}, has not been obtain for any dimension $m$ and without additional hypotheses. Moreover, even with Conjecture \ref{c2} proved, the proof of Conjecture \ref{c1} will be a real challenge, even in the compact case.

It is worth mentioning that if $M$ is a CMC hypersurface, then it is proper biharmonic if and only if $|A|^2=m$ and $f\neq 0$. Therefore, to study CMC proper biharmonic hypersurfaces and give a positive answer to the first conjecture is to study non-minimal CMC hypersurfaces with $|A|^2=m$ and prove that they can only be open parts of $\mathbb{S}^{m}(1/\sqrt{2})$, or $\mathbb{S}^{m_1}(1/\sqrt{2})\times\mathbb{S}^{m_2}(1/\sqrt{2})$, where $m_1+m_2=m$ and $m_1\neq m_2$. This challenging problem can be viewed as a generalization, from the minimal to the CMC case, of the classical result of S. S. Chern, M. do Carmo, and S. Kobayashi, that says that any minimal hypersurface with $|A|^2=m$ is an open part of $\mathbb{S}^{m_1}(r_1)\times\mathbb{S}^{m_2}(r_2)$, where $m_1+m_2=m$, $r^2_1+r^2_2=1$, and $r_1=\sqrt{m_1/m}$ (see \cite{CdCK}).

We will present, in a quite non-exhaustive way, the most important results to date on biharmonic hypersurfaces in spheres. As presentations of this type of results can be also found in \cite{O} and in \cite{ChenOu}, in general, we will focus on the most recent of them.

Conjecture \ref{c1} was proved only in the case of surfaces in \cite{CMO1}. Conjecture \ref{c2} was proved for $m=3$ in \cite{BMO10}. Very recently, in \cite{GLV}, Conjecture \ref{c2} was also proved for $m=4$. The proof of this last result is a long and skilful one and relies on a detailed analysis of the Gauss and Codazzi equations that allow the biharmonic hypothesis. This result actually extends to a non-flat $5$-dimensional space form the one obtained by Y. Fu, M.-C. Hong, and X. Zhan in \cite{FHZ}, which shows that any biharmonic hypersurface in $\mathbb{R}^5$ is CMC and, therefore, minimal.

With geometric or analytical additional hypotheses, the two conjectures have received positive answers in several situations. The geometric hypotheses concern the number of distinct principal curvatures of the hypersurface, or ask the scalar curvature to be constant, or the squared length of the second fundamental form to be bounded by $m$, etc.

Before listing some of these results, we recall that a CMC proper biharmonic submanifold satisfies $|H|\in(0,1]$. In the case of CMC hypersurfaces, like in the more general case of PMC submanifolds, there is a gap in the admissible range of $|H|$.

\begin{theorem}[\cite{BO}]
Let $\phi:M^m\rightarrow\mathbb{S}^{m+1}$ be a CMC proper biharmonic hypersurface with $m>2$. Then $|H|\in(0,(m-2)/m]\cup\{1\}$. Moreover, $|H|=1$ if and only if $\phi(M)$ is an open subset of the small hypersphere $\mathbb{S}^m(1/\sqrt{2})$, and $|H|=(m-2)/m$ if and only if $\phi(M)$ is an open subset of the standard product $\mathbb{S}^{m-1}(1/\sqrt{2})\times\mathbb{S}^1(1/\sqrt{2})$.
\end{theorem}

Among the results answering positively to the two conjectures, when adding the hypothesis on the number of distinct principal curvatures, we mention the following theorem.

\begin{theorem}[\cite{BMO}]\label{BMO2curv}
Let $\phi:M^m\rightarrow\mathbb{S}^{m+1}$ be a proper biharmonic hypersurface with at most two distinct principal curvatures everywhere. Then
$\phi(M)$ is either an open part of $\mathbb{S}^m(1/\sqrt{2})$, or an open part of $\mathbb{S}^{m_1}(1/\sqrt{2})\times\mathbb{S}^{m_2}(1/\sqrt{2})$, $m_1 + m_2 = m$, $m_1\neq m_2$. Moreover, if $M$ is complete, then either $\phi(M)=\mathbb{S}^{m}(1/\sqrt{2})$ and $\phi$ is an embedding, or $\phi(M) =\mathbb{S}^{m_1}(1/\sqrt{2})\times\mathbb{S}^{m_2}(1\sqrt{2})$, $m_1 +m_2 = m$, $m_1\neq m_2$,
and $\phi$ is an embedding when $m_1\geq 2$ and $m_2\geq 2$.
\end{theorem}

As a direct consequence of the above result, for $m=2$ we reobtain the proof of Conjecture \ref{c1} in \cite{CMO1}.

The link between biharmonic hypersurfaces and the isoparametric ones was established by T.~Ichiyama, J. I. Inoguchi, and H.~Urakawa in 2009 and 2010 (see \cite{IIU1}, \cite{IIU}).

\begin{theorem}[\cite{IIU1},\cite{IIU}]\label{teo:isoparametric}
Let $\phi:M^m\rightarrow \mathbb{S}^{m+1}$ be an isoparametric proper biharmonic hypersurface. Then $\phi(M)$ is either an open part of $\mathbb{S}^m(1/\sqrt2)$, or an open part of
$\mathbb{S}^{m_1}(1/\sqrt2)\times\mathbb{S}^{m_2}(1/\sqrt2)$, $m_1+m_2=m$, $m_1\neq m_2$.
\end{theorem}

\noindent The proof relies on the expression of principal curvatures for an isoparametric hypersurface which was then used to solve the equation $|A|^2=m$.

In the case when the number of distinct principal curvatures is less than or equal to $3$ we have two results. The first one holds for $m=3$.

\begin{theorem}[\cite{BMO10}]\label{th:curb_med}
Let $\phi:M^3\to \mathbb{S}^{4}$ be a proper biharmonic hypersurface. Then $\phi$ is CMC.
\end{theorem}

The second result is a generalization of Theorem \ref{th:curb_med} to higher dimensions.

\begin{theorem}[\cite{FuMN}]\label{thfu}
Let $\phi:M^m\rightarrow\mathbb{S}^{m+1}$ be a proper biharmonic hypersurfaces and assume that it has at most three distinct principal curvatures at any point. Then $M$ is CMC.
\end{theorem}

These two results lead to positive answers to Conjecture \ref{c1}, with supplementary hypotheses. First, assuming that the hypersurface is also complete, Theorem \ref{th:curb_med} implies Conjecture \ref{c1}.

\begin{theorem}[\cite{BMO10}]\label{t4.6}
Let $\phi:M^3\rightarrow\mathbb{S}^{4}$ be a complete proper biharmonic hypersurface. Then $\phi(M)=\mathbb{S}^{3}(1/\sqrt 2)$ or $\phi(M)=\mathbb{S}^{2}(1/\sqrt 2)\times\mathbb{S}^{1}(1/\sqrt 2)$.
\end{theorem}

\noindent Indeed, it is easy to verify that, in general, a CMC proper biharmonic hypersurface satisfies $|A|^2=m$ and has positive constant scalar curvature. Therefore, by applying Theorem \ref{th:curb_med}, one obtains that $M^3$ is CMC and has constant scalar curvature. Furthermore, since $M$ is complete, using a result of Q.~M.~Cheng and Q.~R.~Wan in \cite{CW93}, one sees that $M$ is isoparametric and then we conclude with Theorem \ref{teo:isoparametric}.

\begin{remark}
Theorem \ref{t4.6} was actually obtained in \cite{BMO10} under the stronger assumption of compactness, because we used a result of S. Chang in \cite{C93} that shows that any compact CMC hypersurface with constant scalar curvature in $\mathbb{S}^4$ is isoparametric.
\end{remark}

Then, using another of S. Chang's results (see \cite{C94}), Theorem \ref{thfu} implies the next theorem.

\begin{theorem}[\cite{FuMN}]\label{thfu2}
Let $\phi:M^m\rightarrow\mathbb{S}^{m+1}$ be a compact biharmonic hypersurface and assume that it has three distinct curvatures everywhere. Then $M$ is minimal.
\end{theorem}

\begin{remark}
Here we had to slightly modify the original statement of Theorem \ref{thfu2} due to S. Chang's result, that is stated for hypersurfaces with three principal curvatures everywhere (and not for hypersurfaces with at most three principal curvatures everywhere).
\end{remark}

Still following the study of proper biharmonic hypersurfaces according to the number of distinct principal curvatures, Y. Fu and M. C. Hong improved Theorem \ref{thfu} for a greater number of distinct principal curvatures (although considering an additional hypothesis on the scalar curvature). More precisely, using Gauss and Codazzi equations and performing some long computations, they proved the following theorem.

\begin{theorem}[\cite{FH}]
Let $\phi:M^m\rightarrow\mathbb{S}^{m+1}$ be a proper biharmonic hypersurface with at most six distinct principal curvatures at any point. If $M$ has constant scalar curvature, then $M$ is CMC.
\end{theorem}

\begin{remark} The above result holds in any space form $N^{m+1}(c)$, as it is actually given in its original form in \cite{FH}.
\end{remark}

As we have seen the CMC hypothesis implies, for a proper biharmonic hypersurface $M^m$, that $|A|^2=m$ and its scalar curvature is a positive constant. When $M$ is compact, either hypothesis $|A|^2=\cst$, or the scalar curvature is a constant, implies that $M$ is CMC.

The fact that $|A|^2=\cst$ implies that $M$ is CMC, and therefore $|A|^2=m$, is an immediate consequence of two results showing that if the function $|A|^2-m$ has constant sign, then $M$ is CMC. More precisely, we have the following two results.

\begin{theorem}[\cite{BMO12},\cite{O}]\label{prop:|B|>m}
Let $\phi:M^m\rightarrow\mathbb{S}^{m+1}$ be a compact proper biharmonic hypersurface. If $|A|^2\geq m$, then $\phi$ is CMC and $|A|^2=m$.
\end{theorem}

\begin{theorem}[\cite{JHC}]\label{th: jchen1}
Let $\phi:M^m\rightarrow\mathbb{S}^{m+1}$ be a compact proper biharmonic hypersurface. If $|A|^2\leq m$, then $\phi$ is CMC and $|A|^2=m$.
\end{theorem}

The first of the two results follows easily by integrating $\Delta f^2$, which was computed by using the normal part of the biharmonic equation. The second result, on the other side, given by J.~H.~Chen, is more difficult to prove and requires tensor analysis (among others, a Bochner type formula for $\Delta|\grad f|^2$, the expressions of $\Delta f^2$ and $\Delta f^4$ for biharmonic hypersurfaces, and a certain inequality involving $|A|^2$ and $f^2$). Moreover, in \cite{BMO12} and \cite{O}, it was given a proof which is slightly different from the original one, based on the unique continuation property of biharmonic maps.

We mention here that if $|A|^2=\cst$ and $M^m$ is complete and non-compact, then $M^m$ is CMC, under some additional hypotheses involving the Ricci curvature (see \cite{BMO12}, \cite{O}).

The fact that constant scalar curvature implies that $M$ is CMC has recently been proved by S.~Maeta and Y.~L.~Ou (see \cite{MO}).

\begin{theorem}[\cite{MO}]
A compact hypersurface $\phi:M^m\rightarrow\mathbb{S}^{m+1}$ with constant scalar curvature is biharmonic if and only if it is minimal, or it has non-zero constant mean curvature and $|A|^2=m$.
\end{theorem}

\noindent The proof relies, among others, on the same Bochner formula for $\Delta|\grad f|^2$, i.e.,
$$
\frac{1}{2}\Delta|\grad f|^2=\langle\grad\Delta f,\grad f\rangle-\ricci(\grad f,\grad f)-|\nabla\grad f|^2,
$$
and on the general classical inequality
$$
|\nabla\grad f|^2\geq \frac{1}{m}(\Delta f)^2.
$$

In conclusion, we can state the following result.

\begin{theorem}
Let $\phi:M^m\rightarrow\mathbb{S}^{m+1}$ be a compact proper biharmonic hypersurface. Then any of the following hypotheses implies the other two:
\begin{enumerate}

\item[(i)] $M$ is CMC;

\item[(ii)] $|A|^2$ is a constant;

\item[(iii)] the scalar curvature $s$ is a constant.

\end{enumerate}
\end{theorem}

Using the celebrated theorem on CMC hypersurfaces in space forms obtained by K. Nomizu and B. Smyth in \cite{NS}, the next corollary follows immediately.

\begin{corollary}
Let $\phi:M^m\rightarrow\mathbb{S}^{m+1}$ be a compact proper biharmonic hypersurface with constant scalar curvature and $\riem^M \geq 0$, then $\phi(M)=\mathbb{S}^{m}(1/\sqrt{2})$, or $\phi(M)=\mathbb{S}^{m_1}(1/\sqrt{2})\times\mathbb{S}^{m_2}(1/\sqrt{2})$, where $m_1+m_2=m$ and $m_1\neq m_2$.
\end{corollary}

\begin{remark} With the assumption that $M^m$ is CMC instead of constant scalar curvature, the above corollary holds locally (see \cite{BMO12}, \cite{O}).

\end{remark}

Another partial answer to Conjecture \ref{c2}, this time with additional hypothesis of an analytical type, was given in \cite{BLO}, where one proved that a proper biharmonic hypersurface in $\mathbb{S}^{m+1}$ satisfying a certain inequality has the following unique continuation property.

\begin{theorem}[\cite{BLO}]
Let $\phi:M^m\rightarrow\mathbb{S}^{m+1}$ be a proper biharmonic hypersurface. Assume that there exists a non-negative function $h$ on $M$ such that $|\grad|A|^2|\leq h|\grad f|$ on M. If $\grad f$ vanishes on a non-empty open connected subset of $M$, then $\grad f=0$ on $M$, i.e., $M$ has constant mean curvature.
\end{theorem}

\noindent The proof relies again on Aronszajn's result (see \cite{A}), but here the reduction of the order is not necessary anymore. The main idea is to think $\grad f$ as a vector field tangent to $\mathbb{R}^{m+2}$ along $M$ and identify it with an $\mathbb{R}^{m+2}$-valued function $u$. This function satisfies the equation $Lu=f\grad|A|^2$, where $L$ is a second order linear elliptic operator. Then, the result follows using the hypothesis and the Aronszanjn's result.

The inequality in the theorem is verified when one imposes natural geometric conditions on $|A|^2$ or the scalar curvature of $M$.

\begin{corollary}[\cite{BLO}]
Let $\phi:M^m\rightarrow\mathbb{S}^{m+1}$ be a proper biharmonic hypersurface such that $|A|^2$ is constant or $M$ has constant scalar curvature. Then, either $M$ has constant mean curvature, or the set of points where $\grad f\neq 0$ is an open dense subset of $M$.
\end{corollary}

The above corollary is a local version of the results obtained by J. H. Chen, and S. Maeta and Y.-L. Ou, respectively, with the additional hypothesis that $f$ is constant on some open set.

Still with additional hypotheses of an analytic nature (for example, using new Liouville type theorems for superharmonic functions on complete non-compact manifolds, and asking that a certain function on $M$ built by using $f$ to be of class $L^p$, or asking that $f^{-1}\in L^p(M)$, for some $p\in(0,\infty)$), the second conjecture was proved right by S. Maeta in \cite{SM} and Y. Luo and S. Maeta in \cite{LM}.

Another interesting result to prove Conjecture \ref{c1} was obtained by M.~Vieira in \cite{V} by using classical techniques from Riemannian geometry.

\begin{theorem}[\cite{V}]\label{thm:V}
Let $M^m$ be a compact proper biharmonic hypersurface in a closed hemisphere $\mathbb{S}^{m+1}$. If $1-|H|^2$ does not change sign, then $M^m$ is the small hypersphere $\mathbb{S}^m(1/\sqrt{2})$.
\end{theorem}

\noindent The proof of this theorem relies on a general formula that gives the expression of the bilaplacian of a function $f$ on $M$, when $f$ is the restriction of some function $\bar f$ defined on the ambient manifold $N^{m+1}$. This formula is given in terms of the covariant derivatives of $\bar f$ with respect to $N$, and the second fundamental form of $M$ in $N$. A very interesting fact about this formula is the involvement of the tangent and normal parts of the bitension field. Therefore, when $M$ is biharmonic the equation is simpler and can be used in the study of such hypersurfaces. Furthermore, one can consider a given non-zero vector $v\in\mathbb{R}^{m+2}$ which determines the hemisphere containing $M$ in Theorem \ref{thm:V}, and the function $\bar f$ on $\mathbb{S}^{m+1}$ that associates to each point the inner product between its position vector and $v$. Taking into account the properties of this particular function, the bilaplacian of its restriction to $M$ simplifies even more. Finally, integrating over $M$, one concludes.

We end this subsection with yet another two positive partial answers to Conjecture~\ref{c1}.

\begin{theorem}[\cite{JHC}]\label{thmJHC}
Let $\phi:M^m\to \mathbb{S}^{m+1}$ be a compact proper biharmonic hypersurface. If $M^m$ has non-negative sectional curvature and $m\leq 10$, then $\phi$ is CMC and $\phi(M)$ is either $\mathbb{S}^{m}(1/\sqrt 2)$, or $\mathbb{S}^{m_1}(1/\sqrt 2)\times \mathbb{S}^{m_2}(1/\sqrt 2)$, $m_1+m_2=m$, $m_1\neq m_2$.
\end{theorem}

\begin{proof} We present a simpler proof than the original one in \cite{JHC}. By a long but straightforward computation we get a formula for the Laplacian of the squared norm of the shape operator, which holds for any hypersurface $M$,
\begin{equation}\label{wellknown}
\frac{1}{2}\Delta|A|^2=-|\nabla A|^2-m\Div(A(\grad f))+m^2|\grad f|^2-\frac{1}{2}\sum_{i,j=1}^m(\lambda_i-\lambda_j)^2R_{ijij},
\end{equation}
where $\lambda_i$ are the principal curvature functions of $M$.

J. H. Chen proved, in \cite{JHC}, the following inequality that holds for biconservative hypersurfaces in space forms
\begin{equation}\label{prewk}
|\nabla A|^2\geq\frac{m^2(m+26)}{4(m-1)}|\grad f|^2.
\end{equation}
Then, we obtain, for any biconservative hypersurface in a space form, that
\begin{equation}\label{wellknownplus}
\frac{1}{2}\Delta|A|^2\leq\frac{3(m-10)}{m+26}|\nabla A|^2-m\Div(A(\grad f))-\frac{1}{2}\sum_{i,j=1}^m(\lambda_i-\lambda_j)^2R_{ijij}.
\end{equation}

Next, integrating \eqref{wellknownplus}, we prove that $M$ has at most two distinct principal curvatures at any point and, since $M$ is proper biharmonic, from Theorem \ref{BMO2curv}, we conclude.
\end{proof}

The previous theorem deals with compact hypersurfaces, but with the additional hypothesis that the squared norm of the shape operator is constant, we can obtain its local version.

\begin{proposition}\label{p4.20} Let $\phi:M^m\rightarrow\mathbb{S}^{m+1}$ be a proper biharmonic hypersurface with non-negative sectional curvature. If the squared norm of the shape operator is constant and satisfies $|A|^2\leq m$, and $m\leq 6$, then $\phi$ is CMC and $\phi(M)$ is either an open subset of $\mathbb{S}^{m}(1/\sqrt 2)$, or an open subset of $\mathbb{S}^{m_1}(1/\sqrt 2)\times \mathbb{S}^{m_2}(1/\sqrt 2)$, $m_1+m_2=m$, $m_1\neq m_2$.
\end{proposition}

\begin{proof} From the tangent part of the biharmonic equation, we have
$$
A(\grad f)=-\frac{m}{2}f\grad f
$$
and, also using the normal part, $\Div(A(\grad f))$ can be easily written as
$$
\Div(A(\grad f))=\frac{m}{2}(m-|A|^2)f^2-\frac{m}{2}|\grad f|^2.
$$
Replacing in \eqref{wellknown} and using \eqref{prewk}, we get
\begin{equation*}
\frac{1}{2}\Delta|A|^2\leq\frac{5m-32}{m+26}|\nabla A|^2-\frac{m^2}{2}(m-|A|^2)f^2-\frac{1}{2}\sum_{i,j=1}^m(\lambda_i-\lambda_j)^2R_{ijij}\leq 0.
\end{equation*}

Since $|A|^2$ is a constant, it follows that $M$ has at most two distinct principal curvatures at any point, and then, again using Theorem \ref{BMO2curv}, we conclude.
\end{proof}

\begin{remark} The above result can be also seen as a local version of Theorem \ref{th: jchen1}, or of the fact that $|A|^2=\cst$ implies CMC, with additional hypotheses on the curvature and dimension.
\end{remark}

\subsection{Biconservative hypersurfaces}

The study of biconservative hypersurfaces in space forms is a rather new research direction. In one of the first papers devoted to this study, one obtained the following Simons type formula for the squared norm of $S_2$, from the more general equation given in \cite{LO} (see \cite{FLO}).

\begin{theorem}[\cite{FLO}]\label{th:1}
Let $\phi:M^m\rightarrow N^{m+1}(c)$ be a hypersurface in a space form. Then
\begin{align}\label{eq:delta}
\frac{1}{2}\Delta |S_2|^2=&4cm^4f^4-4m^3f^3(\trace A^3)-4m^2f^2|A|^2(cm-|A|^2)\\
\nonumber
&-8m^4f^2|\grad f|^2-4m^2f^2|\nabla A|^2\\
\nonumber
&+4m^2f\langle\grad s,\grad f\rangle\\
\nonumber
&-8m^2\Div(f\ricci(\grad f))-2m^2\Div\left(|A|^2\grad f^2\right)\\
\nonumber
& +\frac{m^5}{8}\Delta f^4-4cm^2(m-1)\Delta f^2-10m^2f\langle\tau_2^{\top}(\phi),\grad f\rangle\\
\nonumber
&-4m^2f^2\Div\left(\tau_2^{\top}(\phi)\right)-2\left|\tau_2^{\top}(\phi)\right|^2+4mf\langle\nabla\tau_2^{\top}(\phi),A\rangle.
\end{align}
\end{theorem}

Theorem \ref{th:1} leads to the next two results.

\begin{theorem}[\cite{FLO}]
Let $\phi:M^m\rightarrow N^{m+1}(c)$ be a constant scalar curvature biconservative hypersurface in a space form. Then
\begin{align}\label{eq:deltaf}
\frac{3m^2}{2}\Delta f^4&=\ 4f^2\big\{cm^2f^2-mf(\trace A^3)-|A|^2(cm-|A|^2)-2m^2|\grad f|^2-|\nabla A|^2\big\}\\\nonumber &=\ 2f^2\left\{-\sum_{i,j=1}^m(\lambda_i-\lambda_j)^2R_{ijij}-4m^2|\grad f|^2-2|\nabla A|^2\right\}.
\end{align}
\end{theorem}

\begin{corollary}[\cite{FLO}]
Let $\phi:M^m\rightarrow \mathbb{S}^{m+1}$ be a biharmonic hypersurface with constant scalar curvature. Then the following system holds
\begin{align}\label{system}
\begin{cases}
\frac{3m^2}{2}\Delta f^4=&4f^2\big\{m^2f^2-mf(\trace A^3)-|A|^2(m-|A|^2)\\
&-2m^2|\grad f|^2-|\nabla A|^2\big\}\\
\quad\quad\Delta f=&f(m-|A|^2).
\end{cases}
\end{align}
\end{corollary}

A direct application of equation \eqref{eq:deltaf} is the following rigidity theorem.

\begin{theorem}[\cite{FLO}]\label{thm:main}
Let $\phi:M^m\rightarrow N^{m+1}(c)$ be a compact biconservative hypersurface in a space form $N^{m+1}(c)$, with $c\in\{-1,0,1\}$. If $M$ is not minimal, has constant scalar curvature and $\riem^M\geq 0$, then $\phi(M)$ is either
\begin{enumerate}

\item[(i)] $\mathbb{S}^m(r)$, $r>0$, if $c\in\{-1,0\}$, i.e., $N$ is either the hyperbolic space $\mathbb{H}^{m+1}$ or the Euclidean space $\mathbb{R}^{m+1}$; or

\item[(ii)] $\mathbb{S}^m(r)$, $r\in(0,1)$, or the product $\mathbb{S}^{m_1}(r_1)\times\mathbb{S}^{m_2}(r_2)$, where $r_1^2+r_2^2=1$, $m_1+m_2=m$, and $r_1\neq\sqrt{m_1/m}$, if $c=1$, i.e., $N$ is the unit Euclidean sphere $\mathbb{S}^{m+1}$.
\end{enumerate}
\end{theorem}

\begin{remark} In \cite{CY}, compact hypersurfaces $M^m$ in $\mathbb{S}^{m+1}$ with $\riem^M\geq 0$ and constant scalar curvature $s\geq m(m-1)$ were classified. In Theorem \ref{thm:main} there is no restriction on the value of scalar curvature.
\end{remark}

\begin{remark}\label{remNS} Rewriting Equation \eqref{eq:delta} in terms of the shape operator $A$ yields a generalization of the well-known formula for CMC hypersurfaces in \cite{NS}. Moreover, Theorem \ref{thm:main} can be viewed as an alternative to the result of K. Nomizu and B. Smyth about compact CMC hypersurfaces in space forms in \cite{NS}, with the "CMC hypothesis" replaced by a more general one about "the biconservativity and constant scalar curvature".
\end{remark}

Imposing the above hypothesis on the curvature of $M^m$ and a restriction on the dimension $m$, but nothing on the scalar curvature, we come to the same conclusion as in Theorem \ref{thm:main}.

\begin{theorem}\label{thmNS}
Let $\phi:M^m\rightarrow N^{m+1}(c)$ be a compact non-minimal biconservative hypersurface. If $\riem^M\geq 0$ and $m\leq 10$, then $\phi$ is CMC and $\phi(M)$ is one of the hypersurfaces given by Theorem \ref{thm:main}.
\end{theorem}

\begin{proof} As we have seen, equation \eqref{wellknownplus} holds for biconservative hypersurfaces. Then, for $m\leq 9$, the right hand side term is non-positive and then, integrating over $M$, we conclude that $\nabla A=0$. 

If $m=10$, we obtain that $M$ has at most two distinct principal curvatures at any point but, using this property, we cannot deduce that $M$ is CMC, as in the biharmonic case. Actually, when $m=10$, we get $\sum_{i,j=1}^{m}(\lambda_i-\lambda_j)^2R_{ijij}=0$ and equality in J. H. Chen's inequality \eqref{prewk}, i.e., $|\nabla A|^2=100|\grad f|^2$. From these two conditions, following the proof of Theorem $4.6$ in \cite{FLO}, we get $\Delta f^2=0$ on $M$. Since $M$ is compact, we obtain that $f$ is constant and thus $\nabla A=0$.
\end{proof}

\begin{remark} In the proof of Theorem \ref{thmNS}, we used the fact that 
$$
-\frac{1}{2}\sum_{i,j=1}^{m}(\lambda_i-\lambda_j)^2R_{ijij}=\langle T,A\rangle,
$$
where $T(X)=-\trace(RA)(\cdot,A,\cdot)$, and a Ricci commutation formula.
\end{remark}

We end this part by presenting a Bochner type formula of independent interest (see \eqref{eq:MSY}) that could be used in further studies on biconservative submanifolds. This formula is inspired by a similar one in \cite{MSY}. It involves a $4$-tensor $Q$ defined on a Riemannian manifold $M$:
$$
Q(X,Y,Z,W)=\langle Y,Z\rangle\langle X,W\rangle-\langle X,Z\rangle\langle Y,W\rangle,
$$
the map
$$
\sigma_{24}(X,Y,Z,W)=(X,W,Z,Y)
$$
which permutes the second and fourth variables, a symmetric $(1,1)$-tensor field $S$, and the $1$-form $\theta$ defined as the contraction $C((Q\circ\sigma_{24})\otimes g^{\ast},\nabla S\otimes S))$, where $g$ denotes the metric on $M$ and $g^{\ast}$ is its dual.

\begin{theorem}[\cite{FLO}]\label{p:MSY} Let $M$ be a Riemannian manifold with the curvature tensor field $R$ and consider a symmetric $(1,1)$-tensor field $S$. Then
\begin{equation}\label{eq:MSY}
\Div\theta=\langle T,S\rangle+|\Div S|^2-|\nabla S|^2+\frac{1}{2}|W|^2,
\end{equation}
where $T(X)=-\trace(RS)(\cdot,X,\cdot)$ and $W(X,Y)=(\nabla_XS)Y-(\nabla_YS)X$.
\end{theorem}

Unlike the Simons type equations, formula \eqref{eq:MSY} extends beyond Codazzi tensors as it involves the antisymmetric part of $\nabla S$.

All the results presented in this section are given for compact biconservative hypersurfaces in $N^n(c)$ and have a certain rigidity. All these hypersurfaces (with additional hypotheses) are given by Theorem \ref{thm:main} and satisfy $\nabla A=0$; in particular, they have at most two distinct principal curvatures at any point.

If we discard the hypothesis on compactness, we get many examples of biconservative hypersurfaces. This fact have been already observed in the section devoted to biconservative surfaces in $N^3(c)$. Then, in the first paper to study biconservative hypersurfaces (\cite{HV95}), Th. Hasanis and Th. Vlachos obtained all these hypersurfaces in $\mathbb{R}^4$, divided in three distinct classes.

In \cite{T2}, N. C. Turgay and A. Upadhyay, studied biconservative hypersurfaces in $N^{m+1}(c)$, obtaining their parametric equation. Then, they explicitly classified these hypersurfaces in $N^4(c)$, $c\neq 0$, that, in addition, have three distinct principal curvatures at any point.

For arbitrary dimensions and for $c=0$, one obtained classification results in \cite{T1} and in \cite{MOR}. In \cite{T1}, one proved that biconservative hypersurfaces with three distinct principal curvatures are either certain rotational hypersurfaces, or certain cylinders. Then, in \cite{MOR}, we performed a detailed qualitative study of biconservative hypersurfaces that are $SO(p+1)\times SO(q+1)$-invariant in $\mathbb{R}^{p+q+2}$, or $SO(p+1)$-invariant in $\mathbb{R}^{p+2}$. This study was done in the framework of the equivariant differential geometry, by using the profile curve associated to the hypersurface in the orbit space. None of these biconservative hypersurfaces are biharmonic.

\section{Open Problems}

Inspired by the recent result in \cite{MP}, that shows that in $\mathbb{S}^3$ there exists an entire family of compact non-CMC biconservative surfaces, we pose the following problem.

\vspace{0.5cm}

\textit{Find examples of compact non-CMC biconservative hypersurfaces in $N^{m+1}(c)$, $m>2$.}

\vspace{0.5cm}

Next, one can see that Remark \ref{remNS} suggests our second open problem.

\vspace{0.5cm}

\textit{Find all biconservative hypersurfaces having constant scalar curvature in space forms.}

\end{document}